\documentclass{article}       % onecolumn (second format)
\usepackage{amsmath,amssymb,amsthm}
\usepackage[margin=1in]{geometry}
\usepackage{graphicx,color}
\usepackage{comment}

\numberwithin{equation}{section}
\newtheorem{definition}{Definition}[section]
\newtheorem{theorem}{Theorem}[section]
\newtheorem{proposition}{Proposition}[section]
\newtheorem{lemma}{Lemma}[section]
\newtheorem{remark}{Remark}[section]

\begin{document}

\title{Critical Behaviour of the Partner Model}

%\titlerunning{Two-Stage Contact Process}

\author{Eric Foxall}

%\institute{
     %         Department of Mathematics and Statistics, University of Victoria \\
	%		  PO BOX 3060 STN CSC Victoria, B.C. Canada V8W 3R4 \\
 %             Tel.: (250) 208-2283\\
%              \email{e.t.foxall@gmail.com}
%}

\maketitle

\begin{abstract}
We consider a stochastic model of infection spread incorporating monogamous partnership dynamics.  In \cite{socon} a basic reproduction number $R_0$ is defined with the property that if $R_0<1$ the infection dies out within $O(\log N)$ units of time, while if $R_0>1$ the infection survives for at least $e^{\gamma N}$ units of time, for some $\gamma>0$.  Here we consider the critical case $R_0=1$ and show that the infection dies out within $O(\sqrt{N})$ units of time, and moreover that this estimate is sharp.
\end{abstract}


\section{Introduction}
The contact process is a well studied stochastic model of infection spread, in which an undirected graph $G=(V,E)$ determines a collection of sites $V$ and edges $E$ which we can think of as individuals and as links between individuals along which the infection can be transmitted.  Each site is either healthy or infectious; infectious sites recover at a certain fixed rate which is usually normalized to $1$, and transmit the infection to each of their neighbours at rate $\lambda$.\\

The contact process has been studied in a variety of different settings, including lattices \cite{speed},\cite{crit},\cite{ips},\cite{sis} (to cite just a few), infinite trees \cite{trees}, power law graphs \cite{plg}, and complete graphs \cite{comp}.  In each case there is a critical value $\lambda_c$ below which the infection quickly vanishes from the graph, and above which the infection has a positive probability of surviving either for all time (if the graph is infinite), or for an amount of time that grows quickly (either exponentially or at least faster than polynomially) with the size of the graph; in the power law case $\lambda_c=0$ so long-time survival is possible whenever $\lambda>0$.\\

In the paper \cite{socon} we introduce a variation of the contact process on the complete graph in which the edges open and close dynamically, modelling the formation and breakup of monogamous partnerships.  In this case the edges $E$ represent \emph{possible} connections and we have a process $\{E_t:t\geq 0\}$ with $E_t\subseteq E$ for each $t\geq 0$, that describes the set of open edges as a function of time.  We call this the \emph{partner model}, and it is defined as follows.\\

There are $N$ individuals, that we picture as vertices on the complete graph $K_N = (V,E)$, \& we denote the set of infectious vertices at time $t$ by $V_t$.  Transmission and recovery are possible, as well as re-infection.  At any moment in time only a subset of the edges are open for transmission, \& the open edges are denoted $E_t$, so the process is $\{(V_t,E_t):t\geq 0\}$.  The transitions are as follows.
\begin{itemize}
\item For each $x \in V_t$, $V_t \rightarrow V_t \setminus \{x\}$ at rate $1$.
\item For each $xy \in E_t$ such that $\{x,y\}\cap V_t = y$, $V_t \rightarrow V_t \cup \{x\}$ at rate $\lambda$.
\item For each $xy \in E$ such that $xz \notin E_t$ and $yz \notin E_t$ for all $z \in V$, $E_t\rightarrow E_t \cup \{xy\}$ at rate $r_+/N$.
\item For each $xy \in E_t$, $E_t \rightarrow E_t \setminus \{xy\}$ at rate $r_-$.
\end{itemize}
Each infectious individual becomes healthy at rate $1$, and along each open edge, an infectious individual infects a healthy individual at rate $\lambda$.  If $x$ and $y$ have no partners, they form a partnership at rate $r_+/N$, and if $xy$ are partnered they break up at rate $r_-$.  The normalization $r_+/N$ is so that each individual finds a partner at a bounded rate.  We can construct the process as a continuous time Markov chain as in \cite{norris}, or else using a graphical construction as described in \cite{socon}.\\

Letting $S_t$ and $I_t$ denote the total number of single (i.e. unpartnered) healthy and infectious individuals respectively, and $SS_t,SI_t,II_t$ the number of partnered pairs of the three possible types, as noted in \cite{socon}, $(S_t,I_t,SS_t,SI_t,II_t)$ is a continuous time Markov chain, whose transition rates can be easily written down.  Defining $s_t=S_t/N$, $i_t = I_t/N$, $ss_t=SS_t/N$, $si_t=SI_t/N$, $ii_t=II_t/N$, $(s_t,i_t,ss_t,si_t,ii_t)$ is a Markov chain as well, and sometimes more convenient to work with.\\

Defining $Y_t=S_t+I_t$ and $y_t=s_t+i_t=Y_t/N$, the proportion of singles, as shown in \cite{socon}, $y_t$ approaches and remains close to a stationary value $y^*$ which is the unique equilibrium in $(0,1)$ for the ODE
\begin{equation*}
y' = r_-(1-y) + r_+y^2
\end{equation*}
To decide whether the infection can spread, as shown in \cite{socon} it suffices to consider the effect of a single infectious individual in an otherwise healthy population, and to track it until the first moment when it either
\begin{itemize}
\item recovers without finding a partner, or
\item if it finds a partner before recovering, breaks up from that partnership.
\end{itemize}

\begin{figure}
\begin{center}
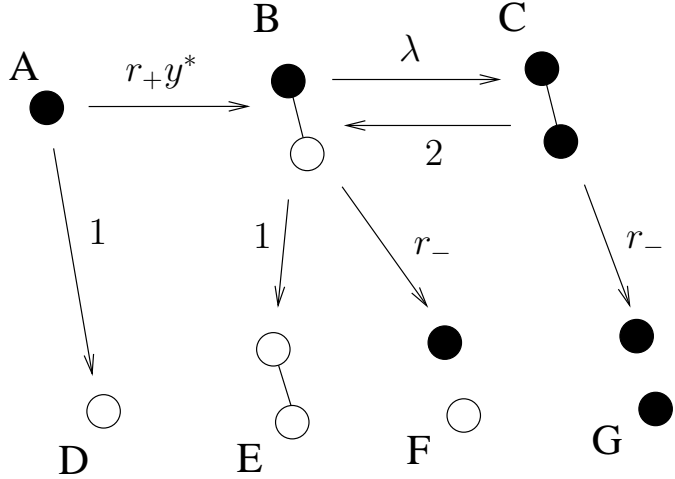
\end{center}
\caption{Markov Chain used to compute $R_0$, with transition rates indicated; shaded sites are infectious}
\label{figri}
\end{figure}

Assuming $y \approx y^*$, this leads to the continuous time Markov chain $(X_t)_{t\geq 0}$ whose transition rates are as shown in Figure \ref{figri}.  Let $\tau = \inf\{t:X_t \in \{D,E,F,G\}\}<\infty$ and define the \emph{basic reproduction number}
\begin{equation}\label{eqri}
R_0 = \mathbb{P}(X_{\tau}=F \mid X_0= A) + 2\mathbb{P}(X_{\tau} = G \mid X_0 = A)
\end{equation}
which is the expected number of infectious singles upon absorption of the above Markov chain, starting from state $A$.  As shown in \cite{socon}, if $R_0<1$ the infection dies out by time $C\log N$, for some $C>0$, with probability tending to $1$ as $N\rightarrow\infty$, while if $R_0>1$ and $|V_0|\geq \epsilon N$ the infection survives up to time $e^{\gamma N}$ with probability $\geq 1-e^{-\gamma N}$ for some $\gamma>0$.  It is shown that if $R_0=1$, then for each $\epsilon>0$, after constant time depending on $\epsilon$, $|V_t|\leq\epsilon N$, however the actual extinction time when $R_0=1$ is not investigated.  Here we prove the extinction time is of order $\sqrt{N}$; the following is the main result.
\begin{theorem}
If $R_0=1$, then
\begin{itemize}
\item there are $C,\gamma>0$ so that for any $(V_0,E_0)$, with probability $\geq 1-e^{-\gamma m}$, $|V_{mC\sqrt{N}}|=0$, and
\item if $|V_0|\geq \sqrt{N}$ and $y_0\geq y^* - \log N/\sqrt{N}$, there is $c>0$ so that $|V_{c\sqrt{N}}| \neq 0$ with probability $\geq 1-e^{-c(\log N)^2}$.
\end{itemize}
\end{theorem}
To get some intuition for this result we suppose $y\approx y^*$, and ignore constants of proportionality in the rates.  Since $R_0=1$, when $i\approx 0$ there is no net drift in the value of $i$, simply a random walk taking steps of size $1/N$ at rate about $Ni$.  For larger values of $i$, the $I+I\rightarrow II$ interactions cause a decrease of about $i$ in the number of new infections produced per partnership, so that we have approximately
$$\frac{d}{dt}i = -i^2$$
The time required to move by an amount $\approx \epsilon i$, when the proportion of single infectious is $i$, is according to the neutral term about $((\epsilon Ni)^2/(Ni) = \epsilon^2 Ni$ and according to the drift term is about $\epsilon i/(i^2) = \epsilon/i$, and the minimum of these two values is maximized when the two are equal, at $i \approx 1/\sqrt{\epsilon N}$, with common value $\epsilon^{3/2}\sqrt{N}$.  This suggests the correct time scale for extinction is $\sqrt{N}$.  To make this rigorous, of course, requires some effort.\\

The proof is laid out as follows.  In Section \ref{secprop}, culminating in Lemma \ref{deltay}, we control the maximum value and the accumulated fluctuations of $y_t-y^*$, which is essential since $y-y^*$ together with $i$ control the rate of spread of the infection at any moment in time.  In Section \ref{secext} we prove the upper bound in four parts.  In each case, we show the probability of the desired event is bounded away from $0$ uniformly in $N$ and in the particular choice of configuration.
\begin{itemize}
\item In Proposition \ref{critsmall} we show the time to go from $N^{\gamma}$ to $0$ infectious is $O(N^{2\gamma})$ for fixed $0<\gamma<1/4$.  The correct order is probably $O(N^{\gamma})$, but $N^{2\gamma}$ suffices here.
\item In Proposition \ref{critbig} we show that for some $C>0$ the time it takes to reach $\leq C\sqrt{N}$ infectious is at most $C\sqrt{N}$.
\item In Proposition \ref{crittopmid} we show that for some $c,C'>0$ the time it takes to reach $\leq c \sqrt{N}$ from $C\sqrt{N}$ infectious is at most $C' \sqrt{N}$.
\item In Proposition \ref{critmid} we show the time it takes to reach $\leq N^{\gamma}$ from $\leq \sqrt{N}$ infectious, for fixed $0<\gamma<1/4$, is $O(\sqrt{N})$.
\end{itemize}
These are combined in Proposition \ref{ub} to prove the upper bound.  The proof of the lower bound is simpler and is given in Proposition \ref{lb}.  The proof of both bounds relies on a useful modification to the definition of the variables $SS,SI,II$, that is given at the start of Section \ref{secext}.

\section{Definitions and Preliminaries}
We begin with a couple of definitions that help us describe the likelihood of important events, and the intervals of time over which they hold.
\begin{definition}\label{defwhp}
An event $A$ holds
\begin{itemize}
\item \emph{with high probability} or whp in $n$ if $\mathbb{P}(A)\geq 1 - e^{-\gamma n}$ for some $\gamma>0$ and $n$ large enough, and
\item \emph{for a very long time in $n$ after $T(n)$} if it holds for $T(n)\leq t \leq e^{\gamma n}$ for some $\gamma>0$ and $n$ large enough.
\end{itemize}
Also, $A$ holds
\begin{itemize}
\item \emph{with good probability} in $n$ if it $\mathbb{P}(A) \geq 1 - e^{-\gamma (\log n)^2}$ for some $\gamma>0$ and $n$ large enough, and
\item \emph{for a long time in $n$ after $T(n)$} if it holds for $T(n) \leq t \leq e^{\gamma (\log n)^2}$ for some $\gamma>0$ and $n$ large enough.
\end{itemize}
In either case, if no ``after $T(n)$'' is mentioned, then $T(n) \equiv 0$.
\end{definition}
Note that both high probability and good probability are preserved under finite intersections.  Also, since $e^{-c(\log n)^2} = n^{-c\log n}$ and $\log n$ is increasing and unbounded, if an event holds with good probability then for any $\alpha>0$ and $n$ large enough, it holds with probability $\geq 1-n^{-\alpha}$.  By the same token, for any $\alpha>0$ and $n$ large enough, ``for a long time in $n$ after $T(n)$'' implies for $T(n) \leq t \leq n^{\alpha}$.\\

\begin{remark}\label{timehorizon}
Since in this paper we are only interested in the behaviour of the process up to the time horizon $O(\sqrt{N})$, if an event holds for a long time (or a very long time) in $N$ then for our sake it holds indefinitely.\\

Moreover, since we want to show only that certain events hold with probability bounded away from $0$ uniformly in $N$, if an event holds with good probability (or high probability) in $N$ -- or indeed with probability $1-o(1)$ as $N\rightarrow\infty$ -- then there is no need to keep track of its probability.
\end{remark}

The following basic large deviations estimate is proved in \cite{socon} and is useful throughout.
\begin{lemma}\label{chern}
Let $X$ be Poisson distributed with mean $\mu$, then
\begin{eqnarray*}
\mathbb{P}(X>(1+\delta) \mu) &\leq & e^{-\delta^2\mu/4} \textrm{ for } 0<\delta\leq 1/2,\\
\mathbb{P}(X<(1-\delta)\mu) &\leq & e^{-\delta^2\mu/2} \textrm{ for }\delta>0\\
\end{eqnarray*}
\end{lemma}

The following result, proved in \cite{socon}, is the starting point for our investigations.
\begin{lemma}\label{start}
For each $\epsilon>0$, there is $T>0$ so that with high probability in $N$ for a very long time after $T$, $|V_t|\leq \epsilon N$ and $|\delta y_t|\leq \epsilon$.
\end{lemma}
In \cite{socon}, for $R_0\neq 1$ we deal with the regime $|V_t|\leq \epsilon N$ using comparison to a branching process, but for $R_0=1$ that approach does not work.  Since the value of $y_t$ affects the rate of spread of the infection, our first step is to get better control on the proportion of singles.

\section{Proportion of Singles}\label{secprop}
In this section we consider the proportion of singles $y_t$, and control its distance from equilibrium, defined as $\delta y := y-y^*$.  We begin by showing that after a little while, $\delta y_t$ has reached a fairly small value.
\begin{lemma}\label{ysquish}
Let $\tau = \inf\{t: |\delta y_t| \leq \log N/\sqrt{N}\}$.  Then there is $C>0$ so that $\tau \leq C\log N$ wgp in $N$.
\end{lemma}
\begin{proof}
In what follows, $c$ refers to a small positive constant and $C$ to a large positive constant that may get smaller (respectively, larger) from step to step.  Moreover, some inequalities hold only for $N$ large enough.  By Lemma \ref{start}, for any fixed $\epsilon>0$ there are $T,\gamma>0$ so that whp $|\delta y_t| \leq \epsilon$ for $T \leq t \leq e^{\gamma N}$, so we may assume $|\delta y| \leq \epsilon$.  Let $ur(y) = r_-(1-y)/2$ and $dr(y) = r_+y^2/2-r_+y/(2N)$ so that $urN$ is the rate of upward transitions in $y_t$ and $drN$ is the rate of downward transitions.  For $y\in [0,1]$, $\max(ur(y),dr(y)) \leq C$ with $C=\max(r_+/2,r_-/2)$ and since $0<y^*<1$, letting $\delta = \min(y^*,1-y^*)/2>0$, for $|y-y^*|\leq \delta$, $\min( ur(y),dr(y)) \geq c - r_+/(2N) \geq c/2$ for $N$ large enough, with $c = \min(r_+/2,r_-/2)\delta^2$.  By definition of $y^*$, $ur(y^*)-dr(y^*) = r_+y^*/(2N)$, so writing $ur$ and $dr$ as functions of $\delta y$, $ur(\delta y)-dr(\delta y) = (ur'(y^*)-dr'(y^*))\delta y + o(\delta y) + r_+y^*/(2N)$, and note that $(ur'(y^*)-dr'(y^*))<0$.  If $|\delta y|N\rightarrow\infty$ as $N\rightarrow\infty$ then $r_+y^*/(2N) = o(\delta y)$; in this case this is satisfied since we may assume $|\delta y| \geq 1/\sqrt{N}$.  Then, for $\epsilon>0$ small enough, $N$ large enough and $|\delta y|\leq\epsilon$, if $\delta y>0$ then $dr - ur \geq c\delta y$ and if $\delta y<0$ then $ur - dr \geq c|\delta y|$.\\

Supposing $\delta y_0 > 0$ and fixing $c,h>0$, let $\tau$ be the first time $t$ such that $\delta y_t \leq (1-2ch(1-ch)/3)\delta y_0$.  Since $ur$ decreases with $y$, if $t<\tau$ then $ur(\delta y_t) \leq ur((1-ch)\delta y_0)$, and since $dr$ increases with $y$, $dr(\delta y_t) \geq dr((1-ch)\delta y_0) \geq ur((1-ch)\delta y_0) + c(1-ch)\delta y_0$.  Using Lemma \ref{chern} with $\mu = ur((1-ch)\delta y_0)hN$ so that $chN \leq \mu \leq ChN$ and $(1+\delta)\mu = \mu+ch(1-ch)\delta y_0 N/3$ so that $\delta = ch(1-ch)\delta y_0 N/(3\mu) \geq c(1-ch)\delta y_0/(3C) = C'\delta y_0$ with $C':=c(1-ch)/3C$, with probability $\geq 1-\exp(-cC'^2h N\delta y_0^2/4))$, either $\tau<h$ or at most $[ur((1-ch)\delta y_0) + c(1-ch)\delta y_0/3]hN$ upward transitions occur in the time interval $[0,h]$.  Using Lemma \ref{chern} with $\mu = dr((1-ch)\delta y_0)hN$ and $(1-\delta)\mu = \mu -c(1-ch)\delta y_0 N/3$, with probability $\geq 1-\exp(-cC'^2hN\delta y_0^2/4)$, either $\tau<h$ or at least $[ur((1-ch)\delta y_0) + 2c(1-ch)\delta y_0/3]hN$ downward transitions occur in the time interval $[0,h]$.  Since each transition moves $\delta y$ by $2/N$, with probability $\geq 1-2\exp(-cC'^2h N \delta y_0^2/4)$, either $\tau<h$ or $\delta y_h \leq \delta y_0 - 2ch(1-ch)\delta y_0/3 = \delta y_0(1-2ch(1-ch)/3)$ which forces $\tau\leq h$.\\

Decreasing $c>0$ if necessary, there are constants $c,h>0$ so that if $\delta y_0 \geq 1/\sqrt{N}$, then with probability $\geq 1-\exp(-cN\delta y_0^2)$ there is $t \in (0,h]$, so that $\delta y_t \leq (1-c)\delta y_0$.  Since the analogous argument applies when $\delta y_0 <0$, we may replace $\delta y_t$ with $|\delta y_t|$ in the previous statement.  So long as $|\delta y|\geq c\log N/\sqrt{N}$ the above probability is at least $1-\exp(-cN(c\log N/\sqrt{N})^2) = 1-\exp(-c(\log N)^2)$, taking $c>0$ smaller if needed.  Since $|\delta y_0| \leq 1$, choosing $k$ so that $|\delta y_0|(1-c)^k \leq \log N/\sqrt{N}$ gives $k \leq C \log N$ with $C = 1/(2\log(1/(1-c)))$.  Thus with probability at least $1-C\log (N) e^{-c (\log N)^2} \geq 1-e^{-(c/2)(\log N)^2}$ for $N$ large enough, there is a time $t \leq Ch\log N$ such that $|\delta y_t| \leq \log N/\sqrt{N}$, as desired.
\end{proof}
The next step is to zoom in on the spatial scale $1/\sqrt{N}$.  Before doing so we collect some facts about random walk on an interval, absorbed at the boundary.
\begin{lemma}\label{gambler}
Let $X_n$ be a (discrete time) random walk on $\{0,...,M\}$, absorbed at $\{0,M\}$, with probability $p_{xx+1}$ of going from $x\rightarrow x+1$ and $p_{xx-1}=1-p_{xx+1}$ of going from $x\rightarrow x-1$, for $x=1,...,M-1$.  Let $T=\inf\{n\geq 0:X_n \in \{0,M\}\}$ and let $b(x) = p_{xx-1}/p_{xx+1}$, then for $x=1,...,M-1$,
\begin{equation*}
\mathbb{P}(X_T = M \mid X_0 = x) = \frac{1+\sum_{j=1}^{x-1}\prod_{i=1}^jb(i)}{1+\sum_{j=1}^{M-1}\prod_{i=1}^jb(i)}
\end{equation*}
with the numerator equal to $1$ when $x=1$.  If $b(x)\equiv b$ is constant in $x$ (although it may depend on $M$) and $M$ is large enough, then for integer $m\geq 1$ and some $C>0$,
\begin{equation*}
\mathbb{P}(T \geq mCM^2) \leq 2^{-m}
\end{equation*}
and if $b=1$, then for each $c>0$, there is $p(c)>0$ with $p(c)\rightarrow 1^-$ as $c\rightarrow 0^+$, such that
\begin{equation*}
\mathbb{P}(T \geq cM^2) \geq p(c)
\end{equation*}
The above two statements remain true if $X_t$ is a continuous time random walk moving at rate $1$.
\end{lemma}

\begin{proof}
The function $h(x) := \mathbb{P}(X_T=M \mid X_0=x)$ satisfies $h(0)=0,h(M)=1$ and for $x \in \{1,...,M-1\}$,
\begin{equation*}
h(x) = p_{xx-1}h(x-1) + p_{xx+1}h(x+1)
\end{equation*}
which rearranged gives $h(x+1) = 1/p_{xx+1}(h(x)-p_{xx-1}h(x-1))$ and subtracting $h(x)$ from both sides while noting $1/p_{xx+1}-1 = p_{xx-1}/p_{xx+1} = b(x)$ gives $h(x+1)-h(x) = b(x)(h(x)-h(x-1))$ and iterating,
\begin{equation*}
h(x+1)-h(x) = (h(1)-h(0))\prod_{j=1}^x b(j) = h(1)\prod_{j=1}^x b(j)
\end{equation*}
and since $h(x+1) = h(1)+\sum_{j=1}^x(h(j+1)-h(j))$,
\begin{equation*}
h(x+1) = h(1)[1+ \sum_{j=1}^x\prod_{i=1}^j b(i)]
\end{equation*}
for $x \in \{1,...,M-1\}$.  Using $h(M)=h(1)[1+\sum_{j=1}^{M-1}\prod_{i=1}^j b(i)]=1$ to solve for $h(1)$ then gives the first expression.\\

If $b(x)\equiv b$ is constant, by symmetry it is enough to consider $b\geq 1$.  In this case a simple coupling argument shows that $X_n$ is stochastically dominated by simple random walk $Y_n$ on $\mathbb{Z}$, provided $X_0 \leq Y_0$.  Letting $Y_0 = X_0$ and $T' = \inf\{n:Y_n=0\}$, $T$ is stochastically dominated by $T'$.  Letting $T_{\pm}$ denote $\inf\{n: Y_n = M/2 \pm M/2\}$, Theorem 3 in \cite{fellerprob} gives, for each $a>0$,
\begin{equation*}
\lim_{M\rightarrow\infty}\mathbb{P}(T_+ \leq aM^2/4) = 2\int_{1/\sqrt{a}}^{\infty}e^{-x^2/2}dx
\end{equation*}
and the same holds, by symmetry, for $T_-$.  Taking $a$ large enough that the right-hand side is $\geq 3/4$, letting $C = a/4$ and noting that $T' = \min(T_+,T_-)$, we find $\mathbb{P}(T' > CM^2) \leq 1/2$ for large enough $M$, and the same holds for $T$ by comparison.  Using the Markov property to iterate gives $\mathbb{P}(T > mCM^2) \leq (1/2)^m$ as desired.  To get the last statement, given $c>0$ let $r(c) = 1-2\int_{1/\sqrt{4c}}^{\infty}e^{-x^2/2}dx>0$ and let $p(c) = r(c)e^{-c}$.  Then, $\lim_{M\rightarrow\infty}\mathbb{P}(T' > cM^2)=r(c)>p(c)$, for large enough $M$ the value is $\geq p(c)$, moreover $p(c)\rightarrow 1^-$ as $c\rightarrow 0^+$.\\

The result in continuous time follows in the same way after controlling the number of transitions up to time $t$ with the help of Lemma \ref{chern}; details are omitted.
\end{proof}

Next we look at a specific Markov chain that as shown later roughly corresponds to the sequence of visits of $\delta y_t$ to the points $\{k/\sqrt{N}:k=1,...\}$.
\begin{lemma}\label{stepchain}
Let $K_n$ be the Markov chain on $\{1,2,...\}$ with $p_{12}=1$ and $p_{kk-1}+p_{kk+1}=1$ with $p_{kk-1}/p_{kk+1}=e^{ck}$ for $k>1$, for some fixed $c>0$.  Let $T = \inf\{n>0:K_n=1\}$ and for $j,k>1$ let $u(j,k) = \mathbb{P}(K_n = k \textrm{ for some }0<n<T \mid K_0=j)$ and let
\begin{equation*}
U(j,k) = \sum_{i=0}^{T-1}\mathbf{1}(K_i = k \mid K_0=j)
\end{equation*}
Then $u(j,k)=1$ if $j>k$, $u(k,k)=p_{kk+1}+p_{kk-1}u(k-1,k)$ and $u(j,k)\leq je^{-c(\binom{k}{2}-\binom{j}{2})}$ if $j<k$, and for $d>0$,
\begin{equation*}
\mathbb{P}(U(j,k)>d) = u(j,k)u(k,k)^d
\end{equation*}
\end{lemma}

\begin{proof}
If $j>k$ then $u(j,k)=1$ since $K_n$ only moves one step at a time, so must visit all of $\{2,..,K_0\}$ before hitting $1$.  The formula for $u(k,k)$ is obvious.  If $j<k$ then letting $T_k = \inf\{n:K_n \in \{1,k\}\}$, $u(j,k) = \mathbb{P}(K_{T_k}=k \mid K_0=j)$ so using Lemma \ref{gambler},
\begin{equation*}
u(j,k) = \frac{1+\sum_{\ell=1}^{j-1}e^{c(\ell+1)\ell/2}}{1+\sum_{\ell=1}^{k-1}e^{c(\ell+1)\ell/2}} \leq je^{(c/2)(j(j-1)-k(k-1))}
\end{equation*}
as desired.  The last statement follows from the definition of $u(j,k)$ and the Markov property.
\end{proof}

Using this result we can control the sum of $K_n$ from the time it starts at level $k$ until it hits level $1$.
\begin{lemma}\label{stepchain2}
For $K_n$ as in Lemma \ref{stepchain} and $j>1$ let $K_0=j$ and let $U(j) = \sum_{n=0}^{T-1} K_n = \sum_{k>1} kU(j,k)$, then if $d$ is large enough,
\begin{equation*}
\mathbb{P}(U(j) \leq j^2d+jd^2+d^3) \geq 1-3e^{-2cd/3}
\end{equation*}
\end{lemma}

\begin{proof}
Let $d_k$ be a sequence of integers, then
\begin{equation*}
\mathbb{P}(U(j) > \sum_{k>1} kd_k) \leq \sum_{k>1} \mathbb{P}(U(j,k) > d_k)
\end{equation*}
so for each $d$, we want a sequence $d_k$ such that the above sum is at most $3e^{-2cd/3}$.  By definition $p_{kk+1}=1/(1+e^{ck}) \leq e^{-ck}$ and from Lemma \ref{stepchain}, $u(k-1,k) \leq (k-1)e^{-c(k-1)}$, so $u(k,k) \leq 2e^{-ck/3}$ uniformly for $k>1$.  Setting $d_k=d$ for $k \leq j$ gives $\mathbb{P}(U(j,k) > d) \leq 2e^{-cdk/3}$.  If $k>j$ then writing $k=j+i$ we find $u(j,j+i) \leq (j+i)e^{-c(ji+i(i-1))} \leq e^{-ci^2}$, so $\mathbb{P}(U(j,k)>d_k) \leq 2e^{-c(i^2 + d_kk/3)}$.  In this case choose $d_k$ so that $i^2+d_kk/3 \leq dk/3$, for which $d_k = d - 3i$ suffices.  Then $\sum_{k>1}\mathbb{P}(U(j,k)>d_k) \leq \sum_{k>1}2e^{-cdk/3} \leq 2e^{-2cd/3}/(1-e^{-cd/3})$ which is $\leq 3e^{-2cd/3}$ for $d$ large enough, and
\begin{equation*}
\sum_{k>1} kd_k = \sum_{k=2}^j kd + \sum_{i=1}^{\lfloor d/3 \rfloor}(j+i)(d-3i) \leq j^2d+jd^2+d^3
\end{equation*}
\end{proof}

Next we control the sum of $K_n$ over repeated excursions away from level $1$.  Before doing so we collect a few useful estimates for some well-known distributions, using the same method as in Lemma \ref{chern}.

\begin{lemma}\label{largedev}
If $X \sim \textrm{binomial}(n,p)$ then for $x >0$ and letting $r = x/np$,
\begin{equation*}
\mathbb{P}(X > x) \leq e^{-x(1/r+\log (r/e))}
\end{equation*}
and for $0<\delta<1$,
\begin{equation*}
\mathbb{P}(X < (1-\delta)np) \leq e^{-np\delta^2/2}
\end{equation*}
If $X_i,i=1,...,m$ are independent and $X_i \sim \textrm{geometric}(p)$ i.e., $\mathbb{P}(X_i>d) = p^d$ for $i=1,...,m$ then letting $S_m =X_1+...+X_m$,
\begin{equation*}
\mathbb{P}(S_m > (1+\delta)m/(1-p)) \leq e^{-m(\delta-\log(1+\delta))}
\end{equation*}
\end{lemma}

\begin{proof}
If $X \sim \textrm{binomial}(n,p)$ then $\mathbb{E}e^{\theta X} = ((1-p)+pe^{\theta})^n = (1+p(e^{\theta}-1))^n \leq e^{np(e^{\theta}-1)}$ so using Markov's inequality, for $\theta\geq 0$, $\mathbb{P}(X>x) \leq e^{-\theta x}\mathbb{E}e^{\theta X} \leq e^{-\theta x+ np(e^{\theta}-1)}$, and setting $x=npr$ with $r=e^{\theta}$ gives
\begin{equation*}
\mathbb{P}(X > x) \leq e^{-x\log r + x - x/r}
\end{equation*}
from which the first estimate follows.  For the lower bound, note for $\theta\geq 0$, $\mathbb{E}e^{-\theta X} \leq e^{np(e^{-\theta}-1)}$ so $\mathbb{P}(X<x) = \mathbb{P}(e^{-\theta X} > e^{-\theta x}) \leq e^{\theta x + np(e^{-\theta}-1)}$.  Setting $x = np(1-\delta)$ with $1-\delta = e^{-\theta}$, the exponent is $-np(1-\delta)\log(1-\delta)-np\delta$, and since $\log(1-\delta) \geq -\delta-\delta^2/2$ for $\delta \leq 1$, $(1-\delta)\log(1-\delta)+ \delta \geq \delta^2/2 + \delta^3/2 \geq \delta^2/2$ and the second estimate follows.\\

If $X \sim \textrm{geometric}(p)$ then $\mathbb{E}e^{\theta X} = (1-p)e^{\theta}/(1-pe^{\theta})$, so if $X_i \sim \textrm{geometric}(p)$ are independent, $i=1,...,m$ and $S_m = X_1 + ... + X_m$ then $\mathbb{E}e^{\theta S_m} = [(1-p)e^{\theta}/(1-pe^{\theta})]^m$, and
\begin{equation*}
\mathbb{P}(S_m > x) \leq e^{-\theta(x-m) + m(\log(1-p) - \log(1-pe^{\theta})}
\end{equation*}
and optimizing in $\theta$ gives $(1-pe^{\theta}) = m/x$ and $\theta = \log((1/p)(1-m/x))$.  Setting $x=(1+\delta)m/(1-p)$ gives $\theta = \log(1+\delta/p) - \log(1+\delta)$ and $\log(1-pe^{\theta}) = \log(1-p)-\log(1+\delta)$ and the exponent $-\log(1+\delta/p)(x-m) +\log(1+\delta)x$ which since $x-m = (\delta+p)m/(1-p)$ is equal to $(-m/(1-p))(p(1+\delta/p)\log(1+\delta/p) - (1+\delta)\log(1+\delta))$.  Now, the function $f(x) := x\log x$ has $f'(x) = 1+\log x$ which increases with $x$, so $f(1+\delta/p)-f(1+\delta) \geq (1/p-1)\delta f'(1+\delta) = (1/p-1)\delta (1+\log (1+\delta))$ and so
\begin{eqnarray*}
p(1+\delta/p)\log(1+\delta/p) - (1+\delta)\log(1+\delta) &=& pf(1+\delta/p)-f(1+\delta)\\
&=& p(f(1+\delta/p)-f(1+\delta)) + (p-1)f(1+\delta)\\
& \geq & (1-p)[\delta (1+\log(1+\delta)) - (1+\delta)\log(1+\delta)]\\
&=& (1-p)[\delta - \log(1+\delta)]
\end{eqnarray*}
and the desired estimate follows.
\end{proof}

\begin{lemma}\label{stepchain3}
Let $T_0=0$ and for $n=1,2,...$ let $T_n = \inf\{m>T_{n-1}:K_m=1\}$.  If $K_0=1$ then there is $C>0$ so that with good probability in $n$,
\begin{equation*}
\sum_{i=0}^{T_n-1}K_i \leq Cn
\end{equation*}
\end{lemma}

\begin{proof}
Let $J=\{j \in \{1,...,n\}:K_{T_{j-1}+1}=2\}$ be the excursions where the first move is up and for $k>1$ let $U_j(2,k) = \sum_{i=T_{j-1}+1}^{T_j-1}\mathbf{1}(K_i=k)$, then
\begin{equation*}
\sum_{i=0}^{T_n-1}K_i = n + \sum_{j \in J,k>1}k U_j(2,k)
\end{equation*}
and for fixed $k$, $\{U_j(2,k):j \in J\}$ are independent and distributed like $U(2,k)$ from Lemma \ref{stepchain}.  Letting $J(k)= \{j \in J:U_j(2,k)>0\}$ be the excursions that reach level $k$, and letting $V(k) = |J(k)$, since $\mathbb{P}(U_j(2,k)>0) = u(2,k)$ and by independence of successive excursions, $V(k) \sim \textrm{binomial}(J,u(2,k))$ and
\begin{equation*}
\sum_{i=0}^{T_n-1}K_i = n + \sum_{k>1}k\sum_{j \in J(k)}U_j(2,k)
\end{equation*}
and for fixed $k$, $\{U_j(2,k):j \in J(k)\}$ are independent and distributed like $\{U(2,k) \mid U(2,k)>0\}$, which has $\mathbb{P}(U(2,k)>d \mid U(2,k)>0) = u(k,k)^d$ i.e., is geometric with parameter $u(k,k)$.  Since $J\leq n$ and
\begin{equation*}
u(2,k) \leq e^{-(c/2)(k(k-1)-2)} \leq e^{-(c/2)((k/2)^2)} = e^{-ck^2/8}
\end{equation*}
it follows that $V(k)$ is at most binomial$(n,p)$ with $p=e^{-ck^2/8}$.  Recalling Lemma \ref{largedev}, to find the values $k$ such that $V(k)=0$ with good probability set $x = 1/2$ so that $r = x/(np) \geq e^{ck^2/8}/(2n)$ and 
\begin{equation*}
\mathbb{P}(V(k) > 0) \leq e^{-(1/2)[ck^2/8 - \log(2ne)]} = (2ne)^{1/2}e^{-ck^2/16}
\end{equation*}
which if $k =\lceil\log n\rceil + k'$ is at most $(2ne)^{1/2}e^{-c(\log n + k')^2/16} \leq (2e)^{1/2}n^{1/2-(c/16)\log n}e^{-ck'^2/16}$ and since $(2e)^{1/2}\sum_{k'\geq 1}e^{-ck'^2/16}<\infty$, for $C>0$ large enough,
\begin{equation*}
\mathbb{P}(V(k)=0 \textrm{ for all }k>\lceil\log n\rceil) \geq 1-Cn^{1/2-(c/16)\log n}
\end{equation*}
which is at most $Cn^{-(c/32)\log n}$, for some possibly larger $C>0$.  For $k=2,...,\lceil \log n \rceil$ and $C>0$ let $x = Cnk^{-3}$, then $r =x/(np) \geq Ck^{-3}e^{ck^2/8}$ and
\begin{equation*}
\mathbb{P}(V(k)>Cnk^{-3}) \leq e^{-nk^{-3}(ck^2/8-3\log k + \log (C/e))}
\end{equation*}
which is at most $e^{-(c/16)n/k}$ for every $k\geq 2$ provided $C$ is taken large enough that $3\log k-\log(C/e) \leq ck^2/16$ uniformly for $k\geq 2$.  Since $k\leq \log n+1 = \log(ne)$, $e^{-(c/16)n/k} \leq e^{-(c/16)n/(\log (ne))}$ and after summing and taking the complement we find
\begin{equation*}
\mathbb{P}(V(k) \leq Cnk^{-3} \textrm{ for }k = 2,...,\lceil\log n\rceil) \geq 1-\log(n)e^{-(c/16)n/(\log(ne))}
\end{equation*}
It remains to control the sums $\sum_{j \in J(k)}U_j(2,k)$.  If $j\in J(k)$ then $U_j(2,k)>0$ and as noted above, $\{U_j(2,k) \mid U_j(2,k)>0\}$ is geometric with parameter $p = u(k,k)$, that is, $\mathbb{P}(U_j(2,k)=d \mid U_j(2,k)>0) = p^{d-1}(1-p)$.  Recalling Lemma \ref{largedev} and setting $\delta=1/2$ while noting $1/2-\log(3/2) \geq 1/16$,
\begin{equation*}
\mathbb{P}(S_m > (3/2)m/(1-p)) \leq e^{-m/16}
\end{equation*}
Thus, on the event that $V(k) \leq Cnk^{-3}$ for each $k \in \{2,...,\lceil\log n\rceil\}$, using the above estimate for each $k$ with $m=Cnk^{-3} \geq Cn/(3\log ne)$ and recalling that $p=u(k,k)$, we find
\begin{equation*}
\mathbb{P}(\sum_{j \in J(k)}U_j(2,k) \leq (3/2)Cnk^{-3}/(1-u(k,k)) \textrm{ for }k=2,...,\lceil\log n\rceil) \geq 1-\log (n) e^{-Cn/(48\log ne)}
\end{equation*}
If $k>1$ then $u(k,k)<1$, and from the proof of Lemma \ref{stepchain2} we have $u(k,k) \leq 2e^{-ck/3}$, so $\sum_{k>1}k^{-2}/(1-u(k,k))<\infty$.  Summarizing, for some $C>0$ the events considered occur together with probability $\geq 1- Cn^{-(c/32)\log n}$, and when they occur,
\begin{equation*}
\sum_{i=0}^{T_n-1} K_i = n + \sum_{j \in J,k>1}kU_j(2,k) \leq n\big{(}1+(3/2)C\sum_{k=2}^{\lceil\log n\rceil}k^{-2}/(1-u(k,k))\big{)} \leq Cn
\end{equation*}
for some possibly larger $C>0$ that does not depend on $n$.
\end{proof}

\begin{lemma}\label{deltay}
Let $\delta y_t = y_t-y^*$ and let $T = \inf\{t:\delta y_t \leq 1/\sqrt{N}\}$, then
\begin{itemize}
\item there is $C>0$ so that with if $N$ large enough, with good probability in $t$, if $|\delta y_0| \leq 1/\sqrt{N}$ then
\begin{equation*}
\int_0^t |\delta y_s| ds \leq Ct /\sqrt{N}
\end{equation*}
\item for each $c>0$, with good probability in $N$, if $|\delta y_0|\leq 1/\sqrt{N}$ then $|\delta y_t|\leq c\log N/\sqrt{N}$ for a long time, while if $|\delta y_0|\leq c\log N/\sqrt{N}$ then $|\delta y_t| \leq 2c\log N/\sqrt{N}$ for a long time, and
\item there is $C'>0$ so that for each $C>0$, if $|\delta y_0| \leq C\log N / \sqrt{N}$ then with good probability in $N$,
\begin{equation*}
\int_0^T |\delta y_t|dt \leq C'(\log N)^6/\sqrt{N}
\end{equation*}
\end{itemize}
\end{lemma}

\begin{proof}
As in the proof of Lemma \ref{ysquish} we use $c$ and $C$ that may change from step to step, and we may assume that $|\delta y_t|\leq \epsilon$ for small $\epsilon>0$.  Recall $Y_t := Ny^*$ and define $\delta Y_t := Y_t - \lfloor Ny^* \rfloor$ whose range is a subset of $\{x \in \mathbb{Z}:|x|\leq N\}$.  Notice that $\delta y -\delta Y/N = \lfloor Ny^* \rfloor/N-y^* = O(1/N)$.  Since $\sqrt{N}\delta y - \delta Y/\sqrt{N} = O(1/\sqrt{N})$, it is enough to show that for some $C>0$, with good probability $\int_0^t |\delta Y_s|ds\leq Ct\sqrt{N}$.  Recalling $ur$ and $dr$ from the proof of Lemma \ref{ysquish}, since $\delta y_t-\delta Y_t/N = O(1/N)$ it follows that for some $c>0$ and $N$ large enough, if $|\delta Y|\geq \sqrt{N}$, then for $\delta Y>0$, $dr - ur \geq c\delta Y/N$ and for $\delta Y <0$, $ur-dr \geq c|\delta Y|/N$.  Since $\max(ur,dr)\leq C$, $dr/ur = (dr-ur)/ur+1 \geq 1+(c/C)\delta Y/N$ when $\delta Y>0$ and similarly $ur/dr \geq 1+(c/C)|\delta Y|/N$ when $\delta Y<0$.  Since $q_+ := Nur$ is the rate of upward transitions and $q_- := Ndr$ is the rate of downward transitions in $\delta Y$, it follows that $q_-/q_+ \geq 1+c\delta Y/N$ when $\delta Y>0$ and $q_+/q_- \geq 1+c|\delta Y|/N$ when $\delta Y<0$.  Letting $M = \lceil\sqrt{N}\rceil$, a simple coupling argument then shows that $|\delta Y_t|$ is stochastically dominated by a continuous time Markov chain $X_t$ on $\{M,M+1,...\}$ with $X_0 = \max(\delta Y_0,M)$, having the same rate of upward transitions as $\delta Y_t$ and satisfying $q_-(M)=0$ and $q_-(x)/q_+(x) = 1+cx/N$ for $x>M$.\\

Let $S=\{jM:j\in \mathbb{Z}\}$ and define a process $Z_t$ on $\{M,M+1,...\}$ as follows.  Let $t_0=\inf\{t>0:Z_t \in S\}$ and define recursively $t_i = \inf\{i>t_{i-1}:Z_t \in \{Z_{t_{i-1}}\pm M\}\}$, and let $K_i = Z_{t_i}/M$, then let $Z_t$ have the same rate of upward transitions as $X_t$ and for $t_{i-1}\leq t <t_i$, have downward transitions with rates satisfying
\begin{equation*}
q_-/q_+ = \begin{cases}1 & \textrm{ if } K_i = 1 \\ 1 + cK_i/(2\sqrt{N}) & \textrm{ if } K_i>1 \end{cases}
\end{equation*}
Comparing the rates one verifies that $X_t$ is stochastically dominated by $Z_t$.  Notice that $K_i$ is a discrete time Markov chain on the state space $\{1,2,...\}$ and satisfies $p_{12}=1$ and $p_{kk-1}+p_{kk+1}=1$ for $k>1$.  Since its transition rates depend on values possibly several steps back in time, $Z_t$ is not, globally in time, a Markov chain, but on time intervals of the form $t \in [0,t_i-t_{i-1})$, the process $\tilde{Z}_t^{(i)} = Z_{t_{i-1}+t}$ is a continuous time Markov chain.  Since $Z_t$ dominates $\delta Y_t$ it is enough to show $\int_0^t Z_sds \leq Ct\sqrt{N}$ with good probability in $t$.  To tackle this integral we break it up as follows: if $t_{n-1} \leq t <t_n$ then 
\begin{equation}\label{flucsum}
\frac{1}{t}\int_0^t Z_s ds \leq \frac{1}{t_{n-1}}\sum_{i=1}^n(t_i-t_{i-1})(K_{i-1}+1)M
\end{equation}
Next we examine what happens on time intervals $[t_{n-1},t_n)$.  For $n>0$, letting $s_0=0$ and $s_i = \inf\{t>s_{i-1}:\tilde{Z}^{(n)}_t \neq \tilde{Z}^{(n)}_{s_{i-1}}\}$ be the jump times of $\tilde{Z}^{(n)}_t$ and defining $\tilde{Z}^{(n)}_i= \tilde{Z}^{(n)}_{s_i}$, as shown in \cite{norris}, $\tilde{Z}^{(n)}_i$ is a discrete time Markov chain with jump probabilities $p_+ = q_+/(q_-+q_+)$ of going up and $p_- = q_-/(q_-+q_+)$ of going down so that $p_-/p_+ = q_-/q_+$, and the random variables $\{s_i-s_{i-1}: n\geq 0\}$ are independent and exponentially distributed with exponential rate $q_++q_-$.  In our case, $cN \leq q_++q_- \leq CN$ and $p_-/p_+ = 1$ if $K_{n-1}=1$ and $=1+cK_{n-1}/(2\sqrt{N})$ if $K_{n-1}>1$, which is constant for fixed $K_{n-1}$ and fixed $N$.  Therefore, if $k >1$ then using the transition probabilities for $\tilde{Z}^{(n)}_i$ and Lemma \ref{gambler}, the transition probability $p_{kk+1}$ for $K_i$ satisfies
\begin{equation*}
p_{kk+1} = \frac{1+\sum_{n=1}^{M-1}(1+ck/(2\sqrt{N}))^n}{1+\sum_{n=1}^{2M-1}(1+ck/(2\sqrt{N}))^n} = \frac{(1+ck/(2\sqrt{N}))^M-1}{(1+ck/(2\sqrt{N}))^{2M}-1} = \frac{1}{(1+ck/(2\sqrt{N}))^M+1}
\end{equation*}
since the denominator is a difference of squares.  Letting $x = (1+ck/(2\sqrt{N}))^M$, $p_{kk+1}=1/(x+1)$ so $p_{kk-1}=x/(x+1)$ and $p_{kk-1}/p_{kk+1}=x$, and since $M=\lceil\sqrt{N}\rceil$, $x\geq e^{ck}$ for some $c>0$ uniformly for $k \in \{2,...,\sqrt{N}\}$.  Therefore, $K_i$ is stochastically dominated by the corresponding chain with $p_{kk-1}/p_{kk+1} = e^{ck}$ defined in Lemma \ref{stepchain}, and so are the quantities $T_n-T_{n-1}$ and $U_j(2,k) = \sum_{i=T_{j-1}}^{T_j-1}\mathbf{1}(K_i=k)$ from the proof of Lemma \ref{stepchain3}.  Rewriting \eqref{flucsum} while focusing on the times $T_n$ gives, for $t_{T_{n-1}} \leq t < t_{T_n}$,
\begin{equation}\label{flucsum2}
\frac{1}{t}\int_0^t Z_s ds \leq \frac{1}{t_{T_{n-1}}}\sum_{k\geq 1}(k+1)\sum_{i=1}^{T_n}(t_i-t_{i-1})\mathbf{1}(K_{i-1} = k)M
\end{equation}
We next consider the values $t_i-t_{i-1}$.  Let $j_0$ be such that $s_{j_0} = t_i-t_{i-1}$.  If $K_{i-1}=1$, the process $\tilde{Z}^{(i)}_t$ is a symmetric simple random walk on $\{0,...,2M\}$ reflected to remain above $M$.  Using Lemma \ref{gambler} with $b=1$, for some $c>0$ and $M$ large enough it follows that
\begin{equation*}
\mathbb{P}(j_0 \geq cM^2) \geq 1/2
\end{equation*}
For other values of $K_{i-1}$, $\tilde{Z}^{(i)}_t$ is a biased random walk on $\{(i-1)M,...,(i+1)M\}$, so using again Lemma \ref{gambler}, for integer $m\geq 1$ and any value of $K_{i-1}$,
\begin{equation*}
\mathbb{P}(j_0\geq CmM^2) \leq 2^{-m}
\end{equation*}
If $K_{i-1}=1$ then since the rate of transitions in $\tilde{Z}^{(i)}_t$ is at most $CN$, using Lemma \ref{chern} with $\mu=cM^2/2$ and $\delta=1/2$, with probability $\geq 1-e^{-cM^2/32}$, at most $3cM^2/4$ transitions have occurred by time $cM^2/(2CN) \geq c/(2C)$, so combining with the above estimate on $j_0$, $\mathbb{P}(t_i-t_{i-1} \geq c/(2C)) \geq 1/2-e^{-cM^2/32}$ which is at least $1/4$ for $M$ large enough.  Since $K_{T_i}=1$ for each $i$, for some $c>0$, the cardinality of the set $\{i\leq n:t_{T_i+1}-t_{T_i}\geq c\}$ is at least $\textrm{binomial}(n,1/4)$.  Then, using Lemma \ref{largedev} with $\delta=1/2$ and the fact that $t_{T_n} \geq \sum_{i=0}^{n-1}t_{T_i+1}-t_{T_i}$,
\begin{equation*}
\mathbb{P}(t_{T_n} \geq cn/8) \geq 1 - e^{-n/32}
\end{equation*}
For any value of $K_{i-1}$, since the rate of transitions in $\tilde{Z}^{(i)}_t$ is at least $cN$, using Lemma \ref{chern} with $\mu = mCM^2$ and $\delta=1/2$ in the other direction, with probability $\geq 1-e^{-mCM^2/16}$, at least $mCM^2/2$ transitions have occurred by time $mCM^2/(cN) \leq mC$ for some possibly larger $C>0$, so $\mathbb{P}(t_i-t_{i-1} \geq mC) \leq 2^{-m}+e^{-mCM^2/16} \leq e^{-cm}$ for some $c>0$ which implies that $t_i-t_{i-1}$ is at most $C(1 + \textrm{geometric}(e^{-c}))$.\\

As shown in the proof of Lemma \ref{stepchain3}, for $i=1,...,T_n$, with good probability in $n$ there are no visits to levels $k>\log n+1$, and there are at most $Cnk^{-3}$ visits to levels $k \in\{2,...,\lceil \log n\rceil\}$.  Using Lemma \ref{largedev} and the above estimate on $t_i-t_{i-1}$, for $k \in \{2,...,\lceil \log n\rceil\}$ and noting $Cnk^{-3}\geq Cn/(3\log n)$, with good probability in $n$, $\sum_{i=1}^{T_n}(t_i-t_{i-1})\mathbf{1}(K_{i-1}=k) \leq C^2(1+(3/2)/(1-e^{-c}))nk^{-3}$ which is at most $Cnk^{-3}$ for some possibly larger $C>0$, and since $K_i=1$ exactly when $i=T_j$ for some $j$, with good probability $\sum_{i=1}^{T_n}(t_i-t_{i-1})\mathbf{1}(K_{i-1}=1) \leq Cn$.  Summing on $k$, with good probability in $n$, $\sum_{k\geq 1}(k+1)\sum_{i=1}^{T_n}(t_i-t_{i-1})\mathbf{1}(K_{i-1}=k) \leq Cn$ for some possibly larger $C>0$ not depending on $n$.  Combining this with the above estimate on $t_{T_n}$ and using \eqref{flucsum2} shows that with good probability in $n$, for $t_{T_{n-1}}\leq t < t_{T_n}$,
\begin{equation}\label{flucsum3}
\frac{1}{t}\int_0^t Z_s ds \leq \frac{8CnM}{c(n-1)} \leq C\sqrt{N}
\end{equation}
for some possibly larger $C>0$; this is nearly enough to prove the first statement but we move onto the second statement for now.  By Lemma \ref{stepchain}, the probability of visiting level $2c\log N$ on any excursion starting from level $c\log N$, is at most $e^{-c'(\log N)^2}$ for some $c'>0$, so with good probability in $N$, after $e^{c'(\log N)^2/2}$ excursions, level $2c\log N$ has still not been visited.  By the above estimate on $t_{T_n}$, with high probability in $e^{c'(\log N)^2/2} \geq N$ for $N$ large enough, this many excursions requires time at least $ce^{c'(\log N)^2/2}/8$, i.e., a long time with respect to $N$.\\

We now prove the third statement. If $Z_0 = \lceil C\log N \rceil M$, we have
\begin{equation*}
\int_0^{t_{T_1}} Z_s ds  \leq \sum_{i=0}^{T_1-1}(t_i-t_{i-1})(K_{i-1}+1)M
\end{equation*}
From the proof of Lemma \ref{stepchain2} it follows that with probability $ \geq 1 - e^{-cd}$ for some $c>0$ there are at most $d$ visits each to levels $2,...,C\log N + d/3$, and no visits to higher levels, before hitting level 1, in which case $T_1 \leq d(C\log N + d/3)$.  Letting $d = (\log N)^2$, $C\log N + d/3$ is at most $(\log N)^2$ for $N$ large enough which gives $T_1 \leq (\log N)^4$, and $e^{-cd} = e^{-c(\log N)^2}$.  Using a similar estimate as above for $\sum_{i=0}^{T_1-1}t_i-t_{i-1}$ and using $K_i +1\leq \log N+d/3+1 \leq (\log N)^2$, with good probability in $N$,
\begin{equation*}
\sum_{i=0}^{T_1-1}(t_i-t_{i-1})(K_{i-1}+1)M \leq C'M(\log N)^6 \leq C'(\log N)^6(\sqrt{N}+1)
\end{equation*}
since $M=\lceil \sqrt{N} \rceil$.  From this, the second statement in the lemma follows.\\

It remains to show \eqref{flucsum3} holds with good probability in $t$.  To do so, first note the random variables $\{T_n-T_{n-1}: n\geq 1\}$ are independent and identically distributed, and that $K_n$ is stochastically dominated by the Markov chain $L_n$ on $\{1,2,...\}$ with $p_{12}=1$, $p_{kk-1}+p_{kk+1}=1$ for $k>1$ and $p_{kk-1}/p_{kk+1} =e^{2c}$, $k\geq 2$,  so defining 
\begin{equation*}
T = \inf\{i>0:L_i = 1 \mid L_0=1\}
\end{equation*}
$T_n - T_{n-1}$ is dominated by $T$.  If $T>i$ then $L_i>1$ and $L_i = 1 + (2X-i)$ where $X \sim \textrm{binomial}(i,p)$ with $p = 1/(1+e^{2c})<1/2$, which implies $X>i/2$.  Using Lemma \ref{largedev} with $x = i/2$ gives $r = x/(ip) = 1/(2p)>1$, so letting $c = 1/r+\log(r/e)$, since $1/r+\log(r/e)$ is positive when $r>1$, $c_0>0$ so for $d \geq 1$, if $i_0$ is taken larger than $1$,
\begin{equation*}
\mathbb{P}(T>i) \leq e^{-ci/2}
\end{equation*}
so $T$ is at most $1 + Y$ where $Y\sim\textrm{geometric}(e^{-c/2})$.  Writing $T_n = \sum_{j=1}^n T_j-T_{j-1}$, taking $C\geq n[1+(1+\delta)/(1-e^{-ci/2})]$ and using again Lemma \ref{largedev} with $\delta=1/2$ while noting $1/2- \log(3/2)\geq 1/16$,
\begin{equation*}
\mathbb{P}(T_n \leq Cn) \geq 1-e^{-n/16}
\end{equation*}
Using the fact that $t_n-t_{n-1}$ is at most $C(1+\textrm{geometric}(e^{-c})$, a similar estimate as for $T_n$ shows that for $N$ large enough, with high probability in $n$, $t_n \leq Cn$, and combining these, with good probability in $n$, $t_{T_n} \leq t_{Cn} \leq Cn$, so if $t_{T_{n-1}}\leq t < t_{T_n}$ then $n \geq t_{T_n}/C > t/C$ and since $x\mapsto e^{-c(\log x)^2}$ is decreasing, $e^{-c(\log n)^2} \leq e^{-c(\log t - \log C)^2}$ which for $t$ large enough is at most $e^{-(c/2)(\log t)^2}$, so that if an event holds with good probability in $n$, then it holds with good probability in $t$.
\end{proof}

\section{Extinction Time}\label{secext}
We return to the partner model, and make a slight modification to the way we record its progress.  First define the three types $SSA_t,SIA_t$ and $IIA_t$, where the $A$ stands for anticipated, as follows.  If a partnership $xy$ is formed at time $t$, let $s>t$ be the first time of breakup of $xy$ after $t$, and at time $t$ record the state of $x$ and $y$ at breakup, and the duration $s-t$ of the partnership.  Then, $SSA_t$ is the number of partnerships at time $t$ that upon breakup will consist of two healthy individuals, and analogously for $SIA_t$ and $IIA_t$.\\

The choice of variables $SSA_t,SIA_t,IIA_t$ may seem unusual, as they are not adapted to the natural filtration $\{\mathcal{F}_t:t\geq 0\}$ of the Poisson point processes that determine the transitions in the model.  However, consider the following modification.  Attach an independent uniform random variable at each partnering event, then use this variable to sample the joint distribution of the final state and duration of the partnership conditioned on its initial state.  This modification preserves the sample path distribution of $(S_t,I_t,SSA_t,SIA_t,IIA_t)$, and the modified process is adapted.  Although we do not make use of this fact, we note the process can be made Markov by introducing ``countdown'' variables for the time of breakup of each partnership, that are incremented at formation, and decrease linearly in time with slope one until breakup occurs.\\

Let $TI_t = I_t + SIA_t + 2IIA_t$ (the $TI$ stands for ``total infectious''), which is the analogue of $|V_t|$ for these new variables.  Since the $SIA\rightarrow S+I$ and $IIA\rightarrow I+I$ transitions leave $TI_t$ unchanged, the only transitions affecting $TI_t$ are the ones affecting infectious singles.  Also, since the final state of a partnership is decided at the moment of partnership formation and recorded in the variables $SSA_t,SIA_t,IIA_t$, the corresponding change in $TI_t$ is felt immediately.  There are three types of transition affecting single infectious:  $I\rightarrow S$, $S+I\rightarrow SI$ and $I+I\rightarrow II$, occurring at respective rates $I$, $(r_+(y-i))I$ and $(r_+i/2)I$.  The second and third type of transition are followed by an immediate transition $SI\rightarrow SSA,SIA$ or $IIA$, and $II\rightarrow SSA,SIA$ or $IIA$, with a probability determined by the Markov chain from Figure \ref{figri}.  At each transition, $TI_t$ can increase by 1, stay the same, or decrease by 1 or 2.  The total rate of transitions affecting $TI_t$ is equal to $I(1 + r_+(y-i)+r_+i/2)$ which we break up into its \emph{principal} part $I(1+r_+(\min(y,y^*)-i) + r_+i/2)$ and its \emph{auxiliary} part $r_+I\max(y-y^*,0) = r_+I\max(\delta y,0)$.\\

First we consider the principal part.  Define $z = 1 + r_+(\min(y,y^*)-i)+r_+i/2$ so that the principal part is $Iz$, and let $p_S = 1/z$, $p_{SI} = r_+(\min(y,y^*)-i)/z$ and $p_{II} = r_+i/(2z)$.  If $y<y^*$ then $z$ increases with $y$ so $p_S$ and $p_{II}$ decrease with $y$, and since $p_S+p_{SI}+p_{II}=1$, $p_{SI}$ increases with $y$.  Referring to the Markov chain $(X_t)_{t\geq 0}$ whose transition rates are depicted in Figure \ref{figri}, recall $\tau = \inf\{t \mid X_t \in \{D,E,F,G\}\}$, and use $\{A \rightarrow F\}$ to denote the event $\{X_{\tau}=F \mid X_0 = A\}$ and similarly for other states.
\begin{itemize}
\item $TI \rightarrow TI+1$ at rate $q_{xx+1} = Izp_{SI}\mathbb{P}(B\rightarrow G)$,
\item $TI \rightarrow TI-1$ at rate $q_{xx-1} = Iz(p_S + p_{SI}\mathbb{P}(B\rightarrow E) + p_{II}\mathbb{P}(C\rightarrow F) + O(1/N))$, and
\item $TI \rightarrow TI-2$ at rate $q_{xx-2} = Iz(p_{II}\mathbb{P}(C\rightarrow E) + O(1/N))$.
\end{itemize}
We take a moment to establish a fact that will be useful a bit later on.
\begin{lemma}\label{lemdrift}
There are $\epsilon,c>0$ so that if $\log N/N \leq i \leq \epsilon$ then for $N$ large enough,
\begin{equation}\label{increst}
q_{xx+1}-q_{xx-1}-2q_{xx-2} \leq -ci(q_{xx+1}+q_{xx-1}+q_{xx-2})
\end{equation}
\end{lemma}
\begin{proof}
Define
\begin{equation}\label{deltaiy}
\Delta(i,y) = p_{SI}\mathbb{P}(B\rightarrow G) -(p_S + p_{SI}\mathbb{P}(B\rightarrow E) + p_{II}\mathbb{P}(C\rightarrow F) + p_{I}) - 2(p_{II}\mathbb{P}(C\rightarrow E))
\end{equation}
Then, it follows that
\begin{equation*}
q_{xx+1} - q_{xx-1}- 2q_{xx-2} = Iz(\Delta(i,y) + O(1/N))
\end{equation*}
In other words, the principal part experiences a drift proportional to $\Delta(i,y)$, $I$ and $z$.  If $i=0$ then $p_{II}=0$ so $p_S+p_{SI}=1$, and if in addition $y\geq y^*$ then $p_{SI} = r_+y^*/(1+r_+y^*) = \mathbb{P}(A\rightarrow B)$, and since $1 = p_S + p_{SI} = p_S + p_{SI}(\mathbb{P}(B\rightarrow E\cup F \cup G)$,
\begin{equation*}
\Delta(0,y) + 1 = 2p_{SI}\mathbb{P}(B\rightarrow G) + p_{SI}\mathbb{P}(B\rightarrow F) = 2\mathbb{P}(A\rightarrow G)+\mathbb{P}(A\rightarrow F) = R_0
\end{equation*}
so if $R_0=1$ then for $y\geq y^*$, $\Delta(0,y)=0$.  Moreover, it is not hard to check that $\partial_i \Delta$ and $\partial_y \Delta$ are well-defined.  In the proof of Lemma 4.2 in \cite{socon} it is shown for $R_0\geq 1$ that $\Delta(i,y^*)$ decreases with $i$, and using the fact that $p_S$ and $p_{II}$ decrease with $y$, exactly the same approach shows that $\partial_y \Delta>0$.  Therefore $\Delta(i,y) \leq \Delta(i,y^*) = -ci + o(i)$ for some $c>0$ and small $i>0$, so for $i>0$ small enough,
\begin{equation*}
\Delta(i,y) \leq -(c/2)i
\end{equation*}
If $i\geq (\log N)/N$, then $O(1/N) = o(i)$ so for some $c>0$ and $N$ large enough, noting that $q_{xx+1}+q_{xx-1}+q_{xx-2} = Iz$ and combining the above observations, the desired result follows.
\end{proof}

If $TI_0$ and $\delta y_0$ are small enough then we can send the process to extinction with positive probability within $o(\sqrt{N})$ amount of time.
\begin{proposition}\label{critsmall}
If $TI_0 \leq N^{\gamma}$ for $\gamma < 1/4$ and $\delta y_0 \leq \log N/\sqrt{N}$ then for some $C,p>0$,
$$\mathbb{P}(|V_{CN^{2\gamma}}|=0) \geq p$$
\end{proposition}

\begin{proof}
Let $T = \inf\{t:TI_t=0 \textrm{ or } TI_t \geq 2N^{\gamma}\}$.  If $t<T$ then since $I_t \leq TI_t \leq 2N^{\gamma}$ and $\delta y_0 \leq \log N/\sqrt{N}$, using both results of Lemma \ref{deltay}, with good probability in $\min(N,t)$ the number of transitions due to the auxiliary part up to time $t$ is at most Poisson with rate $r_+(2N^{\gamma})C((\log N)^6/\sqrt{N} + t/\sqrt{N})$ which is at most $CtN^{\gamma-1/2}$ for large enough $N$, for some $C>0$.  Since for $X \sim \textrm{Poisson}(\lambda)$, $\mathbb{P}(X=0) = e^{-\lambda}$, in this case there are no auxiliary transitions with probability $\geq e^{-CtN^{\gamma-1/2}} \geq 1 - CtN^{\gamma-1/2}$, which for $t \leq CN^{1/2-\gamma-\epsilon}$ is at least $1-CN^{-\epsilon}$.\\

Since an $I+I\rightarrow II$ transition cannot increase $TI_t$, we may ignore these transitions.  Since an $I\rightarrow S$ transition decreases $TI$ by $1$ and an $S+I\rightarrow SI$ transition decreases $TI$ by at most $1$, letting $t_0=0$ and $t_j = \inf\{t:TI_t \neq TI_{t_{j-1}}\}$ be the jump times of $TI_t$ and defining $TI_j = TI_{t_j}$ and $p_S^* = 1/(1+r_+y^*),p_{SI}^* = r_+y^*/(1+r_+y^*)$ while noting the dependence of $p_S$ and $p_{SI}$ on $y$, $TI_j$ is stochastically dominated by the Markov chain with $p_{xx-1} = p_S^* + p_{SI}^*\mathbb{P}(B\rightarrow E)$ and $p_{xx+1} = p_{SI}^*\mathbb{P}(B\rightarrow G)$, whose increments (since $R_0=1$) have expected value $0$, implying $p_{xx-1}=p_{xx+1}=1/2$.  Using Lemma \ref{gambler}, starting from $\leq N^{\gamma}$, with probability $\geq 1/2$ absorption at $0$ or $2N^{\gamma}$ occurs after at most $CN^{2\gamma}$ transitions, so using symmetry, the chain is absorbed at $0$ after at most that many transitions with probability at least $1/4$.  If $|V_t|>0$ then either $I_t>0$ or $IP_t>0$, so either a transition in a single infectious, or breakup of an infectious partnership, is occurring at rate at least $\min(1,r_-)$.  If the transition is of a single infectious then it is a transition in $IT_t$, and if not, the breakup leads to a single infectious, so each transition in $TI_t$ occurs after at most two transitions in either $I_t$ or $IP_t$.  Thus, for $CN^{2\gamma}$ transitions in $TI_t$ to occur requires at most $2CN^{2\gamma}$ transitions in $I_t$ or $IP_t$ which, with high probability in $N^{2\gamma}$, requires time at most $3CN^{2\gamma}\max(1,1/r_-)$.  Summarizing so far, for some $C>0$, with probability $\geq 1/4-O(N^{-\epsilon})$ which is at least $1/8$ for $N$ large enough, $TI_t=0$ for some $t \leq CN^{2\gamma}$.  It remains to control the time until $|V_t|=0$.\\

Given the initial state of a partnership, for a final state $F$ and duration $\tau$ of the partnership, using Bayes' rule for the density functions we have
\begin{equation*}
d\mathbb{P}(\tau \mid F) = \frac{d\mathbb{P}(F \mid \tau)d\mathbb{P}(\tau)}{\mathbb{P}(F)}
\end{equation*}
and since for a given duration each possible final state occurs with probability $\leq 1$, $d\mathbb{P}(F \mid \tau)\leq 1$, and in particular $d\mathbb{P}(\tau \mid F) \leq Cd\mathbb{P}(\tau)$ for some $C>0$ not depending on $F$ which implies $\mathbb{P}(\tau > t \mid F) \leq C\mathbb{P}(\tau >t) = Ce^{-r_-t} = e^{-r_-(t-\log C)}$, so $(\tau\mid F)$ is at most $\log C + \textrm{exponential}(r_-)$, and noting that the time for $n$ particles, each decaying at rate $r$, to all decay is of order $\log n$, with probability $\geq 1/2$, after at most an additional $\log C + C\log N^{\gamma}$ amount of time, $V_t=0$.  Taking $p = (1/8)(1/2) = 1/16$, the result follows.
\end{proof}

Since the rate of transitions in $TI_t$ is a multiple of $I_t$, to get $TI_t$ to decrease quickly enough it would help to know the ratio $I_t/TI_t$ is not too small.  As shown in the next lemma this can be achieved with good probability in $N$ provided $TI \geq (\log N)^2$.
\begin{lemma}\label{proport}
There are $c,h>0$ so that, so long as $TI_t\geq (\log N)^2$, with good probability in $N$, $I_t \geq c TI_t$ for a long time after $h$.
\end{lemma}

\begin{proof}
For any $t\geq 0$, since $S_t$ and $I_t$ are identical in the partner model and the adapted partner model, it follows that $SS_t+SI_t+II_t = SSA_t+SIA_t+IIA_t$.  Moreover, $SSA_t \geq SS_t$ since an $SS$ partnership can't become an $SI$ or an $II$ partnership, which implies that $SIA_t+IIA_t \leq SI_t+II_t = IP_t$ and so $TI_t = I_t + SIA_t + 2IIA_t \leq I_t + 2(SIA_t + IIA_t) \leq I_t + 2IP_t$.  If $I_t \geq c IP_t$ then $I_t + 2IP_t \leq I_t + (2/c)I_t$ and so $I_t \geq TI_t/(1+(2/c))$, so it is enough to show $I_t \geq cIP_t$ for some $c>0$.\\

The rate of transition of each infectious single is at most $1+r_+$.  Using Lemma \ref{chern} with $\mu = (1+r_+)hI_t$ and $\delta=1/2$, from among the $I_t$ infectious singles present at time $t$, whp in $I_t$ at least $(1-(3/2)h(1+r_+))I_t$ of them remain infectious singles over the time interval $[t,t+h]$, and in particular, $\inf_{0\leq s \leq h}I_{t+s} \geq (1-(3/2)h(1+r_+))I_t \geq (1-Ch)I_t$ for some $C>0$ not depending on $h$.\\

On the other hand, the rate of increase of $IP_t$ is at most $r_+I_t$.  Using again Lemma \ref{chern} with $\mu=r_+hI_t$ and $\delta=1/2$, whp in $I_t$, $\sup_{0 \leq s \leq h}IP_{t+s} \leq IP_t + (3/2)r_+h \sup_{0 \leq s \leq h}I_{t+s}$.  Since rate of increase of $I_t$ is at most $2r_-IP_t$, whp in $IP_t$, $\sup_{0\leq s \leq h}I_{t+s} \leq I_t + 3r_-h \sup_{0 \leq s \leq h}IP_{t+s}$, and combining these, whp in $\min(IP_t,I_t)$,
\begin{equation*}
\sup_{0 \leq s \leq h}IP_{t+s} \leq \frac{IP_t + (3/2)r_+hI_t}{1-(9/2)r_+r_-h^2} \leq IP_t(1+Ch) + ChI_t
\end{equation*}
for some $C>0$ not depending on $h$, if $h>0$ is small enough.  Since the rate of increase of $I_t$ is at least $r_-IP_t$ and the rate of transition of each infectious single is at most $1+r_+$, whp in $IP_t$ the number of new (over the interval $[t,t+h]$) infectious singles still present at time $t+h$ is at least $(r_-h/2)(1-(3/2)h(1+r_+)) IP_t \geq ch IP_t$ for some $c>0$ not depending on $h$, if $h>0$ is small enough.\\

Choosing $h>0$ small enough that $Ch \leq 1/2$, for any value of $I_t$, whp in $IP_t$, $I_{t+h} \geq ch IP_t$, and if $I_t \geq ch IP_t$ then whp in $\min(I_t,IP_t)$, $IP_t(1+Ch) \geq \sup_{0\leq s \leq h}IP_{t+s} - I_t/2 \geq \sup_{0 \leq s \leq h}IP_{t+s} - \inf_{0 \leq s \leq h}I_{t+s}$ and $\inf_{0 \leq s \leq h} I_{t+s} \geq I_t/2 \geq ch IP_t/2 \geq (ch/3)\sup_{0 \leq s \leq h} IP_{t+s} - (ch/2)\inf_{0 \leq s \leq h}I_{t+s}$, and so $I_{t+s} \geq (ch/3)/(1+ch/2)IP_{t+s}$ for all $s\in [0,h]$.  An analogous argument shows that for some $c,c'>0$ and any value of $IP_t$, whp in $I_t$, $IP_{t+h} \geq ch I_t$, and if $IP_t \geq ch I_t$ then whp in $\min(I_t,IP_t)$, $IP_{t+s} \geq c'h I_{t+s}$ for all $s \in [0,h]$.  If $h$ is taken small enough then for any value of $IP_t$ and $I_t$, whp in $\max(I_t,IP_t)$, $I_{t+h} \geq ch IP_{t+h}$ and $IP_{t+h} \geq ch I_{t+h}$.  If $TI_t \geq (\log N)^2$ then $\max(I_t,IP_t) \geq (1/3)(\log N)^2$, and ``whp in $(\log N)^2$'' is equivalent to ``wgp in $N$''.  Since after one iteration step $\min(I_t,IP_t) \geq ch \max(I_t,IP_t)$, after one iteration step ``whp in $\min(I_t,IP_t)$'' and ``whp in $\max(I_t,IP_t)$'' are equivalent.  Iterate to get the result wgp in $N$ for a long time after $h$.
\end{proof}

The following result works for more general jump rates with finite range, but is tailored to the present context.
\begin{lemma}\label{hitbound}
Let $X_t$ be a continuous-time jump process on $\{...,M-1,M\}$ absorbed at $\{M\}$ with time-dependent and possibly random jump rates $q_{xx+1},q_{xx-1},q_{xx-2}$ satisfying, almost surely, $0 < q \leq q_{xx+1} + q_{xx-1} + q_{xx-2} \leq Q$ for some $q,Q$ and all $t\geq 0$.   Let $h(x,t) = \mathbb{P}(X_t = M \mid X_0=x)$.  If for some $\lambda>1$ and all $t\geq 0$,
\begin{equation}\label{increst2}
q_{xx+1}(\lambda-1) + q_{xx-1}(\lambda^{-1}-1) + q_{xx-2}(\lambda^{-2}-1) \leq 0
\end{equation}
then $h(x,t) \leq \lambda^{x-M}$ for all $t\geq 0$.\\

Let $X_t$ be as above except on $\{-1,0,...\}$ with $X_0 \in \{0,...,M-1\}$, and let $\tau= \inf\{t:X_t \in \{-1,0\}\}$.  If for all $t\geq 0$, $q_{xx+1}-q_{xx-1}-2q_{xx-2} \leq -b(q_{xx+1}+q_{xx-1}+q_{xx-2})$ then
\begin{equation*}
\mathbb{P}(\tau > 8M/(bq)) \leq e^{-Mb/32}+e^{-cM/(4b)}
\end{equation*}
\end{lemma}

\begin{proof}
For fixed $t$ and $0 \leq s \leq t$ define
\begin{equation*}
h(x,s,t) = \mathbb{P}(X_t = M \mid X_{t-s}=x)
\end{equation*}
so that $h(x,t) = h(x,t,t)$.  An application of the Kolmogorov equations (see \cite{norris}) shows that for $x \in \{1,...,M-1\}$,
\begin{equation*}
\partial_s h(x,s,t) = \sum_{i \in \{-1,1,2\}}q_{xx+i}(t-s)[h(x+i,s,t)-h(x,s,t)]
\end{equation*}
Since $h(x,0,t)=\mathbf{1}(x=M)$, clearly $h(x,0,t) \leq \lambda^{x-M}$ for each $x$, so it is enough to check the condition $h(x,s,t) \leq \lambda^{x-M}$ is preserved as $s$ increases.  Supposing it holds up to time $s$ and letting $q_{xx} = q_{xx+1}+q_{xx-1}+q_{xx-2}$,
\begin{equation*}
\partial_s h(x,s,t) \leq \lambda^{x-M}\sum_{i \in \{-1,1,2\}}q_{xx+i}(t-s)(\lambda^i-1) + q_{xx}(t-s)(\lambda^{x-M}-h(x,s,t))
\end{equation*}
Since the first term is $\leq 0$ by assumption and $q_{xx}\leq Q$, $\partial_s(\lambda^{x-M} - h(x,s,t) \geq -Q(\lambda^{x-M}-h(x,s,t)$, so $\lambda^{x-M}-h(x,s,t) \geq e^{-Qs}(\lambda^{x-M}-h(x,0,t))$, and since $\lambda^{x-M}-h(x,0,t)>0$ for $x \neq M$, in particular $\lambda^{x-M} - h(x,s,t)\geq 0$ for $0 \leq s \leq t$.\\

To get the second result we use large deviations.  Suppose for the moment that $X_t$ has state space $\mathbb{Z}$ and is not absorbed at $\{-1,0,M\}$; for $x$ off the set $\{1,...,M-1\}$ use the value of the jump rates at, say, $x=1$.  Let $t_0=0$ and $t_i = \inf\{t>t_{i-1}:X_t \neq X_{t_{i-1}}\}$ be the jump times of $X_t$, then
\begin{equation*}
\mathbb{P}(X_{t_i} = X_{t_{i-1}}+j) = \int_{t_{i-1}}^{\infty}\frac{\int_{t_{i-1}}^{t_i}q_{xx+j}(s)ds}{\int_{t_{i-1}}^{t_i}q_{xx}(s)ds}d\mathbb{P}(t_i=t)
\end{equation*}
so we find that
\begin{equation*}
\mathbb{E}e^{\theta(X_{t_i}-X_{t_{i-1}})} = \int_{t_{i-1}}^{\infty}\frac{\int_{t_{i-1}}^{t_i}\sum_j e^{\theta j}q_{xx+j}(s)ds}{\int_{t_{i-1}}^{t_i}q_{xx}(s)ds}d\mathbb{P}(t_i=t)
\end{equation*}
then since for $\theta\leq 1/2$ and $|j|\leq 2$, $e^{\theta j} \leq 1+\theta j + j^2$ and using the assumption,
\begin{equation*}
\mathbb{E}e^{\theta(X_{t_i}-X_{t_{i-1}})} \leq \int_{t_{i-1}}^{\infty}\frac{\int_{t_{i-1}}^{t_i}(1 - b\theta + 4\theta^2)q_{xx}(s)ds}{\int_{t_{i-1}}^{t_i}q_{xx}(s)ds}d\mathbb{P}(t_i=t) = 1-b\theta+4\theta^2
\end{equation*}
which if $-b\theta + 4\theta^2\leq 0$ and $b\theta - 4\theta^2$ is small enough is at most $e^{-(b/2)\theta + 2\theta^2}$.  Since the estimate does not depend on the value of $X_{t_{i-1}}$ it follows that $\mathbb{E}e^{\theta(X_{t_n}-X_{t_0})} \leq e^{-n((b/2)\theta - 2\theta^2)}$ so $\mathbb{P}(X_{t_n} > 0 \mid X_0 < M ) \leq e^{-n(((b/2)-M/n)\theta-2\theta^2)}$ which for $\theta = ((b/2)-M/n)/4$ is equal to $e^{-n((b/2)-M/n)^2/8}$, and letting $n = 4M/b$ this is equal to $e^{-nb^2/128} = e^{-Mb/32}$.  By Lemma \ref{chern}, with probability $\geq 1 - e^{-n/16} = 1 - e^{-M/(4b)}$, $t_n \leq 2n/q$, and the result follows.
\end{proof}

By Lemma \ref{start} and noting Remark \ref{timehorizon} we may assume for any fixed $\epsilon>0$ that $|V_t| \leq \epsilon N$ for $t\geq T=T(\epsilon)$.  Since the remainder of our estimates involve times of order $\sqrt{N}$ and $T$ is fixed, we may assume the process begins at time $T$.  Also, since from here till the end of the paper we are only concerned with $TI_t \geq N^{\gamma}$, by Lemma \ref{proport} we may assume $I_t \geq cN^{\gamma}$.  Since $I_t \leq |V_t|$ the conditions of Lemma \ref{lemdrift} are satisfied, so we make use of its conclusion without further comment.\\

Next we show that $TI_t$ can be brought down to $C\sqrt{N}$ by time $C\sqrt{N}$, and with good control on $\delta y_t$, with positive probability.  The approach is to use the drift in the principal part while controlling the contribution from the auxiliary part.\\

\begin{proposition}\label{critbig}
There is $C>0$ so that with probability $\geq 1/2$, for some $t\leq C\sqrt{N}$, $TI_t \leq C\sqrt{N}$ and $|\delta y_t| \leq 2\log N/\sqrt{N}$.
\end{proposition}

\begin{proof}
Using Lemma \ref{ysquish}, with good probability in $N$, for some $t \leq C\log N$, $|\delta y_t|\leq \log N/\sqrt{N}$.  Since $C\log N = o(\sqrt{N})$, we may assume the process began at time $t$, so that $|\delta y_0|\leq \log N/\sqrt{N}$.  By Lemma \ref{deltay}, with good probability in $N$ for a long time in $N$, $|\delta y_t| \leq 2 \log N/\sqrt{N}$, which establishes the easy part of the above statement.  Let $t_0=0$ and $L_0 = TI_0/\sqrt{N}$ then define recursively $t_j = \inf\{t>t_{j-1}: TI_t/TI_{t_{j-1}}\notin [1/2,2]\}$ and $L_j = TI_j/\sqrt{N}$.\\

Given $j\geq 1$ let $X_t$ and $Y_t$ denote the change in $TI_t$ from $t_{j-1}$ up to $t$ due to the principal and auxiliary parts, respectively, so that $X_0=Y_0=0$ and $TI_t = TI_{t_{j-1}} + X_{t-t_{j-1}} + Y_{t-t_{j-1}}$ for $t\geq t_{j-1}$.  By Lemma \ref{proport} we may assume $I_t \geq c\sqrt{N}L_{j-1}/2$ for all $t \in [t_{j-1},t_j]$, so using \eqref{increst}, $X_t$ satisfies the conditions of Lemma \ref{hitbound} with $q = c\sqrt{N}L_{j-1}/2$ and $b = cL_{j-1}/(2\sqrt{N})$.  We use the following fact to estimate $t_j$ as well as the value of $TI_{t_j}$: if $|Y_t| \leq w\sqrt{N}L_{j-1}$ and $X_t \leq -(1/2+w)L_{j-1}\sqrt{N}$ then $t\geq t_j-t_{j-1}$.\\

Recall that $Y_t$ has transitions at rate $r_+I_t\max(\delta y_t,0)$.  Letting $t^* = m\sqrt{N}/L_{j-1}$ and using Lemma \ref{deltay} and $|\delta y_0| \leq \log N/\sqrt{N}$, with good probability in $\min(N,m\sqrt{N}/L_{j-1})$ which since $L_{j-1}$ will always be larger than some constant, is $\geq m\sqrt{N}/L_{j-1}$, for some $C>0$
\begin{equation*}
\sup_{t \leq t^*}|Y_t| \leq C\sqrt{N}L_{j-1}((\log N)^6 + m\sqrt{N}/L_{j-1})/\sqrt{N} \leq C\sqrt{N}L_{j-1}((\log N)^6/\sqrt{N} + m/L_{j-1})
\end{equation*}
Assuming that $t^* \geq t_j-t_{j-1}$ and that the right-hand side is at most $w\sqrt{N}L_{j-1}$, then using \eqref{increst} and applying the second result of Lemma \ref{hitbound} with $b,q$ as above and $M = ((3/2)+2w)\sqrt{N}L_{j-1}$, since $bM = -(3+4w)L_{j-1}^2/4$, with probability $\geq 1 - e^{-cL_{j-1}^2} - e^{-cN}$ for some $c>0$, $t_j-t_{j-1}\leq C'\sqrt{N}/L_{j-1}$ for some $C'>0$.  We see the above assumption is satisfied by taking $m = C'$ and $L_{j-1}$ and $N$ large enough that $(\log N)^6/\sqrt{N}+m/L_{j-1} \leq w/C$.\\

If we set $\lambda=1+\epsilon$ in the first statement in Lemma \ref{hitbound}, then $\lambda-1=\epsilon$, $\lambda^{-1}-1 = -\epsilon + O(\epsilon^2)$ and $\lambda^{-2}-1 = -2\epsilon + O(\epsilon^2)$, so if $\epsilon>0$ is small enough and $q_{xx+1}-q_{xx-1}-2q_{xx-2} \leq -\epsilon(q_{xx+1}+q_{xx-1}+q_{xx-2})$ then \eqref{increst2} is satisfied.  In this case, the hypothesis is satisfied with $\epsilon = cL_{j-1}/(2\sqrt{N})$, so using $x = \sqrt{N}(L_{j-1}(1/2+w)$ and $M=(3/2-2w)\sqrt{N}L_{j-1}$ and $\lambda = 1 + \epsilon$ the hitting probability of the lower endpoint is at least $1- (1+cL_{j-1}/(2\sqrt{N})^{-(1-3w)\sqrt{N}L_{j-1}} \geq 1 - e^{-cL_{j-1}^2}$ for some $c>0$ provided $w<1/3$, which can be prescribed.\\

We now estimate the probability that starting from $TI_0 \leq \epsilon N$ for $\epsilon>0$ to be chosen, $TI_{t_j} \leq TI_{j-1}/2$ repeatedly until $TI_t \leq C\sqrt{N}$.  To go from $\epsilon N$ down to $C\sqrt{N}$ by multiples of $2$ requires at most $\log_2 N$ steps.  At each step the desired event has probability $\geq 1 - 2e^{-cL^2} - e^{-cN}$ minus the error term in ``good probability in $\sqrt{N}/L$'', which is at most $e^{-c(\log(\sqrt{N}/L))^2}$ for some $c>0$.  Since $TI_t \leq 2|V_t|$ we may assume $TI_t \leq \epsilon$ for $t\geq 0$ which gives $L_j \leq \epsilon \sqrt{N}$ for all $j$, so taking $\epsilon = c/C$ and then summing over values of $L$ such that $C/c \leq L \leq \epsilon \sqrt{N}$ in decreasing order in the first case and in increasing order in the second case, the probability is at least
$$1 - 2\sum_{j=0}^{\infty}e^{-C^22^{2j}/c} - \log_2 N e^{-cN} - \sum_{j=0}^{\infty}e^{-c(\log(2^jC/c))^2}$$
which is at least $1/2$ for $C$ large enough.  On this event the time taken is $m\sqrt{N}/L_{j-1}$ at each step, where $m$ is fixed, which summing over the above values is at most $(mc\sqrt{N}/C)\sum_{j=0}^{\infty}2^{-j} \leq 2mc\sqrt{N}/C$.
\end{proof}

Next we show that for fixed $C>c$, $TI_t$ can be brought down from $C\sqrt{N}$ to $c\sqrt{N}$ with positive probability within time $C'\sqrt{N}$ for some $C'>0$.
\begin{proposition}\label{crittopmid}
For any $C>c>0$, there are $C',p>0$ so that with probability $\geq p>0$ uniformly in $N$, if $TI_0 \leq C\sqrt{N}$ and $|\delta y_0| \leq 2\log N/\sqrt{N}$ then for some $t \leq C'\sqrt{N}$, $TI_t \leq c\sqrt{N}$ and $|\delta y_t| \leq 4 \log N/\sqrt{N}$.
\end{proposition}

\begin{proof}
The proof of the estimate on $|\delta y_t|$ is the same as in Proposition \ref{critbig}.  As in the proof of Proposition \ref{critbig} let $X_t$ and $Y_t$ denote the contributions due to the principal and auxiliary parts, in this case starting from $t=0$.  As observed in the proof of Proposition \ref{critsmall}, the discrete time jump chain for $X_t$ is dominated by a symmetric simple random walk $\tilde{X}_n$ with $\tilde{X}_0=0$.  Setting  $T = \inf\{n:\tilde{X}_n= \pm M/2\}$ and using Lemma \ref{gambler}, for any $a>0$ and large enough (even) $M$, $\mathbb{P}(T \leq aM^2) \geq p(a) >0$ and by symmetry, $\mathbb{P}(T \leq aM^2,\,\,\tilde{X}_T = -M/2) \geq p(a)/2$.  Since $TI_t \geq c\sqrt{N}$ on the region of interest, using Lemma \ref{proport} we may assume $I_t \geq cc'\sqrt{N}$ for some $c'>0$, which implies the rate of transition of $X_t$ is at least $c'\sqrt{N}$ for some possibly smaller $c'>0$.  Setting $M/2 = \lceil C\sqrt{N} \rceil$ and comparing to $\tilde{X}_n$, with probability at least $p(a)/2 - e^{-c'N}$ for some $c'>0$, for some $t \leq aC^2N/(c'\sqrt{N}) = aC'\sqrt{N}$, either $TI_t \leq c\sqrt{N}$ or $X_t \leq -C\sqrt{N}$.\\

Following the proof of Proposition \ref{critbig}, with good probability in $\sqrt{N}$, for $t^* = aC'\sqrt{N}$,
$$\sup_{t \leq t^*}|Y_t| \leq C''\sqrt{N}((\log N)^6/\sqrt{N} + aC')$$ 
which if $a$ is taken small enough, is at most $c\sqrt{N}$ when $N$ is large.  Since $X_t \leq -C\sqrt{N}$ and $|Y_t| \leq c\sqrt{N}$ forces $TI_t \leq c\sqrt{N}$, the proof is complete.
\end{proof}

Now, we start from $TI_0 \leq c\sqrt{N}$ for small $c>0$ and show that $TI_t$ can be brought down to $N^{\gamma}$ for some $\gamma<1/4$ within time $C\sqrt{N}$.

\begin{proposition}\label{critmid}
There are $c,C>0$ so that if $TI_0 \leq c\sqrt{N}$ and $|\delta y_0| \leq 4\log N/\sqrt{N}$ then with probability $\geq 1/2$ there is $t\leq C\sqrt{N}$ so that $TI_t \leq N^{\gamma}$ for some $\gamma<1/4$.
\end{proposition}

\begin{proof}
Define $S,t_j,L_j$ and $Y_t$ as in the proof of Proposition \ref{critbig}, and this time define $X_t$ as follows.  Since $I+I\rightarrow II$ transitions do not cause $TI_t$ to increase, they can be ignored.  Break up the rate of $I\rightarrow S$ transitions into two parts $I_t/(1+r_+(\min(y,y^*)-i))$ and $I_tr_+(\min(y,y^*)-i)/(1+r_+(\min(y,y^*)-i))$ then the first part together with the principal part of $S+I\rightarrow SI$ transitions gives expected change $0$ and the remaining part gives change $-1$ which can also be ignored.  What remains is a simple random walk, that we denote $X_t$, moving at rate at least $cI_t$ for some $c>0$.  This is similar to the observation made in Proposition \ref{critsmall}, except that here we have a continuous time walk to work with, which in this case is more convenient.  By Lemma \ref{proport}, we may assume the transition rate is at least $cTI_t$ for some possibly smaller $c>0$.  We have $X_0=Y_0=0$ and $TI_t \leq TI_{t_{j-1}} + X_{t-t_{j-1}} + Y_{t-t_{j-1}}$ for $t\geq t_{j-1}$.\\

Recall that if $|Y_t| \leq w\sqrt{N}L_{j-1}$ and $X_t \leq -(1/2+w)\sqrt{N}L_{j-1}$ then $t\geq t_j-t_{j-1}$.  Using Lemma \ref{gambler} (actually, using the one-sided first passage time discussed in the proof) and dividing by the transition rate, for $N$ large enough, with probability $\geq 1-2^{-m}$ the first time when $X_t$ satisfies the above bound is at most $mCNL_{j-1}^2/(c\min_{t_{j-1}\leq t \leq t_j}TI_t)\leq mC\sqrt{N}L_{j-1}$ for some possibly larger $C>0$.  As before, letting $t^* = mC\sqrt{N}L_{j-1}$, with good probability in $\min(N,mC\sqrt{N}L_{j-1})$, which for $TI\geq N^{\gamma}$ is at least good probability in $N$,
\begin{equation}\label{auxest}
\sup_{t \leq t^*}|Y_t|\leq (\sqrt{N}L_{j-1})C'((\log N)^6 +m\sqrt{N}L_{j-1})/\sqrt{N} \leq mC\sqrt{N}L_{j-1}^2
\end{equation}
for some $C>0$, which is at most $w\sqrt{N}L_{j-1}$ provided $mL_{j-1} \leq (w/C)$.  On this event, and on the event that $X_t$ dips below $-(1/2+w)\sqrt{N}L_{j-1}$ before going above $(1-w)\sqrt{N}L_{j-1}$, and all before time $mC\sqrt{N}L_{j-1}$, it holds that $TI_{t_j} \leq TI_{t_{j-1}}/2$.  The probability this occurs is at least $(1-w)/((1/2+w)+(1-w)) = 2(1-w)/3$, minus $2^{-m}$, minus the error term in ``good probability in $N$''.  Therefore if $w$ is small enough and $m$ is large enough, which can be prescribed, then for $N$ large enough the probability is at least $3/5$, say.  Note the bound on $|Y_t|$ is good so long as $L_{j-1} \leq w/(mC)$, which holds provided we start from $TI_0 \leq c\sqrt{N}$ with $c$ small enough and $L_j \leq L_0$ for all $j$ under consideration.  In this case the values $k_j$ defined by $k_0=0$ and recursively by $k_j = k_{j-1} \pm 1$ according as $L_j \geq 2L_{j-1}$ or $L_j \leq L_{j-1}/2$ are dominated by a random walk with $p_{xx-1}=3/5$ and $p_{xx+1}=2/5$, which we now consider.\\

Letting $K_+ = 1$ and $K_- = \inf\{k:2^{-k} \leq TI_0/N^{\gamma}\}$ and setting $J = \min\{j:k_j \in \{K_-,K_+\}\}$, then letting $\tau = \inf\{t: TI_t \leq N^{\gamma}\}$, we have $\tau \leq \sum_{j=1}^J (t_j-t_{j-1})$.  Using Lemma \ref{gambler} with $M = K_+-K_-$, $b=3/2$ and $x=M-1$,
\begin{equation*}
\mathbb{P}(k_J = K_-) = \frac{1+(3/2)+...+(3/2)^{x-1}}{1+(3/2)+...+(3/2)^x}= \frac{2}{3}\frac{1-(2/3)^x}{1-(2/3)^{x-1}} \geq \frac{2}{3}
\end{equation*}
Letting $U(k) = \sum_{j=0}^{J-1} \mathbf{1}(k_j = k \mid k_0 = K_+-1)$, applying the above calculation at level $k$ and using the Markov property we find $\mathbb{P}(U(k)>d) \leq (2/3)^d$.  Summing, we find
\begin{equation*}
\mathbb{P}(U(K_+-j) \leq jd \textrm{ for }j \in \{1,...,K_+-K_--1\}) \geq 1 - \sum_{m=1}^{\infty}(2/3)^{md} = 1 - \frac{(2/3)^d}{1-(2/3)^d}
\end{equation*}
On this event, and on the event the time spent at each visit to level $K_+-j$ is at most $jdm\sqrt{N}L_{j-1}$,
\begin{equation*}
\tau \leq m\sqrt{N}\sum_{j=1}^{K_+-K_--1}(jd)^22^{-j} \leq mC'd^2 \sqrt{N}
\end{equation*}
for some $C'>0$ not depending on $N$ or $d$.  To estimate the time taken, start from any value of $TI_t \in  [TI_{t_{j-1}}/2,2TI_{t_{j-1}}]$ and let $t^* =t +  m\sqrt{N}L_{j-1}$, then note that in the same way as above, with probability $\geq p>0$ for some $p$, $X_{t^*}-X_t \leq -(3/2+2w)\sqrt{N}L_{j-1}$ and $\sup_{t \leq s \leq t^*}|Y_s-Y_t| \leq w\sqrt{N}L_{j-1}$, which for any value of $TI_t$ in the above range ensures that $t_j\leq t^*$.  Applying this repeatedly and combining with the other estimate, altogether the desired event has probability
\begin{equation*}
1 - \left(\frac{(2/3)^d}{1-(2/3)^d} + \sum_{j=0}^{\log_2 N}(jd)(1-p)^{jd} \right)
\end{equation*}
which tends to $1$ as $N,d\rightarrow\infty$, and in particular, for some $d$ is at least $1/2$ when $N$ is large enough.
\end{proof}
It is now easy to show that when $R_0=1$, the disease dies out by time $C\sqrt{N}$ with positive probability.
\begin{proposition}\label{ub}
There are $C,\gamma>0$ so that from any initial distribution of $(V_0,E_0)$, for $N$ large enough and integer $m$,
\begin{equation*}
\mathbb{P}(V_{mC\sqrt{N}}=0) \geq 1-e^{-\gamma m}
\end{equation*}
\end{proposition}
\begin{proof}
Let
\begin{itemize}
\item $t_1 = \inf\{t:TI_t \leq C\sqrt{N},\,\,|\delta y_t| \leq \log N/\sqrt{N}\}$,
\item $t_2 = \inf\{t>t_1:TI_t \leq c\sqrt{N},\,\,|\delta y_t| \leq 2 \log N/\sqrt{N}\}$,
\item $t_3 = \inf\{t>t_2: TI_t \leq N^{\gamma},\,\,|\delta y_t| \leq 4 \log N\sqrt{N}\}$ and
\item $t_4 = \inf\{t>t_3: |V_t|=0\}$.
\end{itemize}
Apply Propositions \ref{critbig}, \ref{crittopmid}, \ref{critmid}, and \ref{critsmall} in that order, and use the Markov property at each step, to deduce that $t_4 \leq C'\sqrt{N}$ with probability $\geq p>0$ for some $p,C'$ uniformly in $N$.  To get the above statement, let $1-p = e^{-\gamma}$ and apply the Markov property repeatedly.
\end{proof}
We conclude with a matching lower bound that works when $|V_0| \geq \sqrt{N}$ and $\delta y_0\leq -\log N/\sqrt{N}$.
\begin{proposition}\label{lb}
If $|V_0| \leq \sqrt{N}$ and $\delta y_0 \geq -\log N/\sqrt{N}$ there is $c>0$ so that with good probability in $N$,
\begin{equation*}
V_{c\sqrt{N}} \neq 0
\end{equation*}
\end{proposition}

\begin{proof}
If $|V_0| \geq f(N)$ then $\max(I_0,IP_0) \geq f(N)/3$.  If $I_0 \geq f(N)/3$ then $TI_0 \geq f(N)/3$, and if $IP_0 \geq f(N)/3$ then since an infectious partnership has a positive probability of breaking up before both individuals are healthy, and since simultaneous partnerships' final states are independent, for some $c>0$, with high probability in $f(N)$, $TI_0 \geq f(N)/c$.  In this case take $f(N)=\sqrt{N}$, and let $\tau = \inf\{t:TI_t \leq TI_0/2\}$.\\

To get a bound in the other direction we redefine the principal and auxiliary parts as follows.  The principal part is as before except with $\max(y,y^*)$ instead of $\min(y,y^*)$.  The auxiliary part consists of $S+I\rightarrow SI$ type events at rate $\max(-\delta y,0)$.  In this case, the auxiliary part is a component of the principal part, so letting $X_t$ denote the effect of the principal part and $Y_t$ the effect of the auxiliary part to time $t$, $TI_t = TI_0 + X_t - Y_t$.\\

Break up the $\max(y,y^*)$ in the principal part into $y^*$ and $\max(y-y^*,0)$, then the transition probabilities corresponding to the $y^*$ part are fixed and have expected value $\Delta(i,y^*)$ which is at least $-ci$ for some $c>0$ and thus $\geq -c/\sqrt{N}$ for some $c>0$ for $t < \tau$, and the transition probabilities corresponding to the $\max(y-y^*,0)$ part are fixed and have expected value $\geq 0$.  Since the rate of transitions is $O(\sqrt{N})$, the average rate of decrease of $TI$ is at most $O(1)$, so using a large deviations argument on the $y^*$ part, and one on the $\max(y-y^*,0)$ part, shows that with high probability in $\sqrt{N}$ the decrease due to the principal part up to time $c\sqrt{N}$ is at most $cC\sqrt{N}$, for some $C>0$.\\

It is easy to check, as in the proof of Proposition \ref{critmid}, that for $t^*=c\sqrt{N}$, with good probability in $\sqrt{N}$, $\sup_{t \leq t^*}|Y_t|$ is at most $cC\sqrt{N}$ for some $C>0$.  By taking $c$ small enough, with good probability in $\sqrt{N}$, $|X_t|+|Y_t| \leq TI_0/3$ for $t\leq c\sqrt{N}$ and so $\tau \geq c\sqrt{N}$.
\end{proof}

\section*{Acknowledgements}
The author wishes to thank Chris Hoffman for the suggestion to study the model on the complete graph.  The author's research is partly supported by an NSERC PGSD2 Graduate Scholarship.

\bibliography{twoStage}
\bibliographystyle{plain}
\end{document}